\documentclass[11pt]{amsart}
\usepackage{fullpage}
\usepackage{amssymb,amsmath,amsfonts,amsthm,mathrsfs}
\usepackage[all]{xy} \SelectTips{cm}{10}
\usepackage{color}

\numberwithin{equation}{section}

\renewcommand{\AA}{\mathbb A}

\newcommand{\FF}{\mathbb F}

\newcommand{\PP}{\mathbb P}
\newcommand{\QQ}{\mathbb Q}
\newcommand{\RR}{\mathbb R}

\newcommand{\ZZ}{\mathbb Z} 
\newcommand{\Zhat}{\widehat\ZZ}

\newcommand{\OO}{\mathcal O}
\newcommand{\calG}{G}  
\newcommand{\calJ}{\mathcal J}
\newcommand{\calB}{\mathcal B}

\newcommand{\aA}{\mathfrak a} 
\newcommand{\p}{\mathfrak p}
\newcommand{\Pp}{\mathfrak P} 

\newcommand{\norm}[1]{ \left|\!\left| #1 \right|\!\right|  }

\def\cyc{{\operatorname{cyc}}}
\def\ab{{\operatorname{ab}}}
\def\Spec{\operatorname{Spec}} 
\def\Gal{\operatorname{Gal}}
\def\ord{\operatorname{ord}} 
\def \GL {\operatorname{GL}}  

\def \SL {\operatorname{SL}}
\def \PSL {\operatorname{PSL}}
\def\Aut{\operatorname{Aut}} 

\def\Frob{\operatorname{Frob}}
\newcommand{\tors}{{\operatorname{tors}}}
\def\pr{\operatorname{pr}}

\def\bbar#1{\setbox0=\hbox{$#1$}\dimen0=.2\ht0 \kern\dimen0 
\overline{\kern-\dimen0 #1}}
\newcommand{\Qbar}{{\overline{\mathbb Q}}} 
 
\newcommand{\kbar}{\bbar{k}} 
\newcommand{\FFbar}{\overline{\FF}} 

\newtheorem{thm}{Theorem}[section]
\newtheorem{lemma}[thm]{Lemma}

\newtheorem{prop}[thm]{Proposition}

\theoremstyle{definition}
\newtheorem{definition}[thm]{Definition}

\theoremstyle{remark}
\newtheorem{remark}[thm]{Remark}

\newenvironment{romanenum}{\hfill \begin{enumerate} }{\end{enumerate}}
\newenvironment{alphenum}{\hfill \begin{enumerate} }{\end{enumerate}}

\definecolor{webcolor}{rgb}{0.8,0,0.2}
\definecolor{webbrown}{rgb}{.6,0,0}
\usepackage[
        colorlinks,
        linkcolor=webbrown,  filecolor=webcolor,  citecolor=webbrown, 
        backref,
        pdfauthor={David Zywina}, 
       pdftitle={Elliptic curves with maximal Galois action on their torsion points},
]{hyperref}
\usepackage[alphabetic,backrefs,lite]{amsrefs} 

\begin{document}
\title[]{Elliptic curves with maximal Galois action on their torsion points}
\tableofcontents 
\author{David Zywina}
\address{Department of Mathematics, University of Pennsylvania, Philadelphia, PA 19104-6395, USA}
\email{zywina@math.upenn.edu}
\urladdr{http://www.math.upenn.edu/\~{}zywina}
\date{\today}

\maketitle

\begin{abstract}
Given an elliptic curve $E$ over a number field $k$, the Galois action on the torsion points of $E$ induces a Galois representation,  $\rho_E \colon \Gal(\overline{k}/k) \to \GL_2(\widehat{\mathbb{Z}}).$  
For a fixed number field $k$, we describe the image of $\rho_E$ for a ``random'' elliptic curve $E$ over $k$.  In particular, if $k\neq \mathbb{Q}$ is linearly disjoint from the cyclotomic extension of $\mathbb{Q}$, then $\rho_E$ will be surjective for ``most'' elliptic curves over $k$.
\end{abstract}

\section{Introduction}
Let $E$ be an elliptic curve defined over a number field $k$.  For each positive integer $m$, let $E[m]$ be the group of $m$-torsion points in $E(\kbar)$.    The group $E[m]$ is isomorphic to $(\ZZ/m\ZZ)^2$ and is acted on by the absolute Galois group $G_k:=\Gal(\kbar/k)$.  This action may be expressed in terms of a Galois representation,
\[
\rho_{E,m} \colon \calG_k \to \Aut(E[m]) \cong \GL_2(\ZZ/m\ZZ).
\]
Taking the inverse limit over all $m$ (ordered by divisibility), we have a Galois representation
\[
\rho_E \colon \calG_k \to \Aut(E(\kbar)_\tors) \cong \GL_2(\Zhat).
\]
%
In this paper, we will describe how large the image of $\rho_E$ is for a ``random'' elliptic curve over $k$.  \\

A renowned theorem of Serre (\cite{Serre-Inv72}*{Th\'eor\`eme~3}) says that if $E$ does not have complex multiplication, then $\rho_E(G_k)$ has finite index in $\GL_2(\Zhat)$.  In the same paper, Serre shows that $\rho_E(\calG_\QQ)\neq \GL_2(\Zhat)$ for every elliptic curve $E$ over $\QQ$.

The first example of a surjective $\rho_E$ was given in the doctoral thesis of A.~Greicius \cite{GThesis}.  Let $\alpha\in \Qbar$ be a root of the polynomial $x^3+x+1$, and let $E/\QQ(\alpha)$ be the elliptic curve given by the Weierstrass equation $y^2+2xy+\alpha y = x^3 - x^2$.  In \cite{GThesis}*{\S3.3}, it is  shown that $\rho_E(\calG_{\QQ(\alpha)})=\GL_2(\Zhat)$.    

\subsection{Statement of results}

Denote the ring of integers of $k$ by $\OO_k$.  For $(a,b) \in \OO_k^2$, define $\Delta_{a,b} = -16(4a^3+27b^2)$.  If $\Delta_{a,b}\neq 0$, then let $E(a,b)$ be the elliptic curve over $k$ defined by the Weierstrass equation 
\[
Y^2=X^3+aX+b.
\]   
Fix a norm $\norm{\cdot}$ on $\RR\otimes_\ZZ \OO_k^2 \cong \RR^{2[k:\QQ]}$.   For each real number $x>0$, define the set
\[
B_k(x) = \{ (a,b) \in \OO_k^2 : \Delta_{a,b} \neq 0,\, \norm{(a,b)}\leq x \}.
\]
Thus to each pair $(a,b)\in B_k(x)$, we can associate an elliptic curve $E(a,b)$ over $k$. 
Note that 
\begin{equation} \label{E:box asymptotics}
|B_k(x)|\sim \kappa x^{2[k:\QQ]}
\end{equation} 
as $x\to \infty$, where $\kappa>0$ is a constant which depends on $k$ and $\norm{\cdot}$.

The following theorem addresses a question of Greicius on the surjectivity of the $\rho_E$ (\cite{GThesis}*{\S3.4 Problem~3}).  Let $\QQ^\cyc\subseteq \kbar$ be the cyclotomic extension of $\QQ$.
\begin{thm} \label{T:main theorem 1}
Suppose that $k\cap \QQ^{\cyc} =\QQ$ and $k\neq \QQ$.  Then
\[
\lim_{x\to \infty} \frac{ |\{ (a,b) \in B_k(x) :  \rho_{E(a,b)}(\calG_k) = \GL_2(\Zhat) \}| }{ |B_k(x)| } = 1.
\]
\end{thm}
Intuitively, the theorem says that for a randomly chosen pair $(a,b)\in \OO_k^2$, the corresponding elliptic curve $E(a,b)$ satisfies $\rho_{E(a,b)}(\calG_k)=\GL_2(\Zhat).$  \\

Let $\chi_k\colon \calG_k \to \Zhat^\times$ be the cyclotomic character of $k$.  For each elliptic curve $E$ over $k$, we have $\det \circ \rho_E=\chi_k$.  In particular, the assumption $k\cap \QQ^{\cyc} =\QQ$ (equivalently $\chi_k(\calG_k)=\Zhat^\times$) is necessary for Theorem~\ref{T:main theorem 1}.  For a number field $k$, define the group
\[
H_k := \{ A \in \GL_2(\Zhat) : \det(A) \in \chi_k(\calG_k)\}.
\]
Given an elliptic curve $E$ over $k$, we certainly have $\rho_E(\calG_k)\subseteq H_k$.   Our main theorem, which generalizes Theorem~\ref{T:main theorem 1}, shows that this is the only general constraint for $k\neq \QQ$.
\begin{thm} \label{T:main theorem 2}
Let $k\neq \QQ$ be a number field.  Then
\[
 \frac{ |\{ (a,b) \in B_k(x) :  \rho_{E(a,b)}(\calG_k) \neq H_k \}| }{ |B_k(x)| } 
\ll_{k,\norm{\cdot}} \frac{\log x}{\sqrt{x}}.
\]
\end{thm}

\begin{remark}
Theorem~\ref{T:main theorem 2} shows that the proportion of $(a,b)$ in $B_k(x)$ with $\rho_{E(a,b)}(\calG_k)=H_k$, as a function of $x$,  quickly approaches $1$.  The implicit constant in the theorem is effective and depends only on $k$ and the fixed norm.
\end{remark}

Before continuing, let us introduce some more notation.  Let $E$ be an elliptic curve over a field $k$.  For each positive integer $m$,
denote the fixed field in $\kbar$ of $\ker(\rho_{E,m})$ by $k(E[m])$.

\subsection{The rationals} \label{SS:rationals}
For completeness, we now discuss the case $k=\QQ$ which was excluded from Theorem~\ref{T:main theorem 2}.  Let $E$ be an elliptic curve over $\QQ$, and let $\Delta$ be the discriminant of some Weierstrass model of $E$ over $\QQ$.   There exists an integer $n\geq 1$ such that $\QQ(\sqrt{\Delta}) \subseteq \QQ(\mu_n)$, where $\mu_n$ is the set of $n$-th roots of unity (the assumption $k=\QQ$ is important here!). 

Using that the field $\QQ(\sqrt{\Delta})$ lies in both $\QQ(E[2])$ and $\QQ(\mu_n)\subseteq \QQ(E[n])$, Serre deduced that the index $[\GL_2(\ZZ/2n\ZZ): \rho_{E,2n}(\calG_\QQ)]$ is \emph{even} (for details, see \cite{Serre-Inv72}*{pp.~310-311}) and in particular $\rho_E(\calG_\QQ)\neq \GL_2(\Zhat)$.  Following Lang and Trotter, we make the following definition.
\begin{definition}
An elliptic curve $E$ over $\QQ$ is a \emph{Serre curve} if $[\GL_2(\Zhat) : \rho_{E}(\calG_\QQ) ]=2$.  
\end{definition}

A Serre curve is thus an elliptic curve $E$ over $\QQ$ for which $\rho_E(\calG_\QQ)$ is as large as possible.  For a Serre curve $E/\QQ$, the group $\rho_{E}(\calG_\QQ)$ can be described explicitly in terms of the field $\QQ(\sqrt{\Delta})$.

N.~Jones \cite{Jones} (building on work of Duke \cite{Duke}) has shown that ``most'' elliptic curves over $\QQ$ are Serre curves.  The analogue of Theorem~\ref{T:main theorem 2} is then the following.  
\begin{thm}[Jones]  \label{T:Jones-Serre curves} There is a constant $\beta>0$ such that 
\[
 \frac{ |\{ (a,b) \in B_\QQ(x) : E(a,b) \text{ is not a Serre curve} \}| }{ |B_\QQ(x)| } \ll_{\norm{\cdot}} \frac{(\log x)^\beta}{\sqrt{x}}.
\]
\end{thm}
\begin{remark}
Theorem~\ref{T:Jones-Serre curves} is a special case of \cite{Jones}*{Theorem~4} and will be proven in \S\ref{SS:Serre curves}.    Unlike Jones' version, the constants in our proof will be effective.
\end{remark}

\subsection{Overview of proof} 
Suppose that $E$ is an elliptic curve over a number field $k\neq \QQ$.  There is an exact sequence
\[
1\to \SL_2(\Zhat) \to \GL_2(\Zhat) \stackrel{\det}{\to} \Zhat^\times \to 1,
\]
and the representation $\det \circ \rho_E\colon \calG_k \to \Zhat^\times$ is the cyclotomic character $\chi_k$ of $k$.  Therefore,
\[
\rho_E(\calG_k) \cap \SL_2(\Zhat)  = \rho_E(\calG_{k^\cyc}).
\]
Thus the equality $\rho_E(\calG_k)=H_k$ is equivalent to $\rho_E(\calG_{k^\cyc}) = \SL_2(\Zhat)$.  A group theoretic argument will show that this in turn is equivalent to having $\rho_{E,m}(\calG_{k^\cyc})=\SL_2(\ZZ/m\ZZ)$ whenever $m$ is equal to $4$, $9$, or a prime $\geq5$.

For a prime $m=\ell\geq 5$, the condition $\rho_{E,\ell}(\calG_{k^\cyc})=\SL_2(\ZZ/\ell\ZZ)$ is equivalent to $\rho_{E,\ell}(\calG_{k})\supseteq \SL_2(\ZZ/\ell\ZZ).$   By considering the Frobenius endomorphism for the reduction of $E$ modulo several primes $\p\subseteq \OO_k$, we can determine which conjugacy classes of $\GL_2(\ZZ/\ell\ZZ)$ meet $\rho_{E,\ell}(\calG_k)$.    Combining this modulo $\p$ information together, we will use the large sieve to give an asymptotic upper bound for the growth of
\[
|\{(a,b) \in B_k(x) : \rho_{E(a,b),\ell}(\calG_k)\not\supseteq \SL_2(\ZZ/\ell\ZZ) \}|
\]
as a function of $x$; see \S\ref{S:exceptional primes}.  To understand the distribution of reductions modulo $\p$, we will use a recent result of Jones; see \S\ref{S:over finite fields}.  Of significant importance is a theorem of Masser and W{\"u}stholz, which is needed to bound the number of $\ell$'s that must be considered.

The conditions at $m=4$ or $9$ are more involved.  In particular, for $m=4$ we will need to impose the condition that $\sqrt{\Delta}$ is not in the cyclotomic extension of $k$ (this avoids the obstruction of \S\ref{SS:rationals} that always occurs for $k=\QQ$).  In \S\ref{S:discriminants}, we again use the large sieve to bound the number of $(a,b)\in B_k(x)$ for which $\sqrt{\Delta_{a,b}}$ (and $\sqrt[3]{\Delta_{a,b}}$ if $\mu_3\subseteq k$) lie in the cyclotomic extension of $k$.

Our main theorems will then be deduced in \S\ref{S:proof}.

\subsection*{Acknowledgments} 
Many thanks to Bjorn Poonen for his careful reading of this paper and helpful comments.   Thanks also to Aaron Greicius and Nathan Jones.

\subsection*{Notation and conventions}   \label{S:notation}
For each field $k$, let $\kbar$ be an algebraic closure of $k$ and let $\calG_k := \Gal(\kbar/k)$ be the absolute Galois group of $k$.  For each integer $n\geq 1$, let $\mu_n$ be the group of $n$-th roots of unity in $\kbar$.   Let $k^\cyc$ (resp.~$k^\ab$) be the cyclotomic (resp.~maximal abelian) extension of $k$ in $\kbar$.   

For a number field $k$, denote its ring of integers by $\OO_k$.  Let $\Sigma_k$ be the set of non-zero prime ideals of $\OO_k$.  For each $\p\in \Sigma_k$, we have a residue field $\FF_\p=\OO_k/\p$; let $N(\p)$ be its cardinality.    Let $\Sigma_k(x)$  be the set of primes $\p$ in $\Sigma_k$ with $N(\p)\leq x$.  Let $\ord_\p\colon k^\times \to \ZZ$ be the surjective discrete valuation corresponding to $\p$.   Denote the absolute discriminant of $k$ by $d_k$.

Fix a group $G$.  Let $G'$ be the derived subgroup of $G$, i.e., the minimal normal subgroup of $G$ for which $G/G'$ is abelian.  Equivalently, $G'$ is the group generated by the set $\{xyx^{-1}y^{-1}:x,y\in G\}$.  The {abelianization} of $G$ is $G^\ab:=G/G'$.  Profinite groups will always be considered with their profinite topologies.

For $a,b$ in a field $k$, define $\Delta_{a,b}= -16(4a^3+27b^2)$.  If $\Delta_{a,b}\neq 0$, then let $E(a,b)$ be the elliptic curve over $k$ defined by the Weierstrass equation $Y^2=X^3+aX+b$.

Suppose that $f$ and $g$ are real valued functions of a real variable $x$.  By $f\ll g$ (or $g \gg f$), we mean that there are positive constants $C_1$ and $C_2$ such that for all $x\geq C_1$, $|f(x)| \leq C_2 |g(x)|$.  We use $O(f)$ to represent an unspecified function $g$ with $g \ll f$.   The dependencies of the implied constants will always be indicated by subscripts.  Also, all implicit constants occurring in this paper are effective.

Finally, the symbols $\ell$ and $p$ will always denote rational primes.

\section{Criterion for maximal Galois action} \label{S:criterion}

 \begin{prop} \label{P:Criterion} 
Let $E$ be an elliptic curve over a number field $k$, and let $\Delta$ be the discriminant of a Weierstrass model of $E$ over $k$.  
Suppose that the following conditions hold:
\begin{alphenum}
\item \label{I:a} $\rho_{E,\ell}(\calG_k)\supseteq \SL_2(\ZZ/\ell\ZZ)$ for every prime $\ell\geq 5$,
\item \label{I:b} $\rho_{E,4}(\calG_k)\supseteq \SL_2(\ZZ/4\ZZ)$ and $\rho_{E,9}(\calG_k)\supseteq \SL_2(\ZZ/9\ZZ)$,
\item \label{I:c} $\sqrt{\Delta} \not\in k^\cyc$,
\item \label{I:d} $\mu_3 \not\subseteq k$ or $\sqrt[3]{\Delta}\not\in k^\cyc$.
\end{alphenum}
Then $\rho_E(\calG_k)=H_k$.
\end{prop}

\begin{remark}
\begin{romanenum}
\item
The image of $\Delta$ in $k^\times/(k^\times)^{12}$ depends only on the isomorphism class of $E/k$.  Thus for a positive integer $r$ dividing $12$, the $r$-th root of $\Delta$,  up to a factor in $\mu_r\cdot k^\times$, is independent of all choices.   In particular, conditions (\ref{I:c}) and (\ref{I:d}) are well-defined.
\item
The Kronecker-Weber theorem says that $\QQ^\cyc=\QQ^\ab$, so condition (\ref{I:c}) \emph{never} holds for $k=\QQ$.
\end{romanenum}
\end{remark}

Since $\det\circ \rho_E \colon \calG_k \to \Zhat^\times$ is the cyclotomic character of $k$, we find that $\rho_E(\calG_{k})=H_k$ if and only if $\rho_E(\calG_{k^\cyc})=\SL_2(\Zhat)$.  Applying Lemma~\ref{L:SL(Zhat) Goursat} to $\rho_E(\calG_{k^\cyc})$, we have $\rho_E(\calG_{k})=H_k$  if and only if $\rho_{E,m}(\calG_{k^\cyc})=\SL_2(\ZZ/m\ZZ)$ holds whenever $m$ is $4$, $9$, or a prime $\geq 5$.  Proposition~\ref{P:Criterion} is then an immediate consequence of the following lemma.

\begin{lemma} \label{L:cyclotomic exceptional}
Let $E$ be an elliptic curve over a number field $k$ with discriminant $\Delta\in k^\times/(k^\times)^{12}$.
\begin{romanenum}
\item  
Let $\ell\geq 5$ be a prime.  If $\rho_{E,\ell}(\calG_k) \supseteq \SL_2(\ZZ/\ell\ZZ)$, then $\rho_{E,\ell}(\calG_{k^\cyc}) = \SL_2(\ZZ/\ell\ZZ)$.
\item    
If $\rho_{E,4}(\calG_k) \supseteq \SL_2(\ZZ/4\ZZ)$ and $\sqrt{\Delta} \not\in k^\cyc$, then $\rho_{E,4}(\calG_{k^\cyc})=\SL_2(\ZZ/4\ZZ)$.
\item    
If $\rho_{E,9}(\calG_k) \supseteq \SL_2(\ZZ/9\ZZ)$ and $\sqrt[3]{\Delta} \not\in k^\cyc$, then $\rho_{E,9}(\calG_{k^\cyc})=\SL_2(\ZZ/9\ZZ)$.
\item    
If $\rho_{E,9}(\calG_k) \supseteq \SL_2(\ZZ/9\ZZ)$ and $\mu_3 \not\subseteq k$, then $\rho_{E,9}(\calG_{k^\cyc})=\SL_2(\ZZ/9\ZZ)$.
\end{romanenum}
\end{lemma}
\begin{proof}
Let $m$ be a positive integer such that $\rho_{E,m}(\calG_k)\supseteq \SL_2(\ZZ/m\ZZ)$.  Since $k^\cyc$ is an abelian extension of $k$, we have inclusions
\begin{equation} \label{E:cyclotomic inclusions}
\SL_2(\ZZ/m\ZZ)' \subseteq \rho_{E,m}(\calG_{k})' \subseteq \rho_{E,m}(\calG_{k^\cyc}) \subseteq \SL_2(\ZZ/m\ZZ).
\end{equation}
\noindent (i) Suppose that $m=\ell\geq 5$ is prime.   By Lemma~\ref{L:derived subgroup finite} we have $\SL_2(\ZZ/\ell\ZZ)'=\SL_2(\ZZ/\ell\ZZ)$, so from (\ref{E:cyclotomic inclusions}) we deduce that $\rho_{E,\ell}(\calG_{k^\cyc}) = \SL_2(\ZZ/\ell\ZZ)$.  

\noindent (ii)  Our assumption $\rho_{E,4}(\calG_k) \supseteq \SL_2(\ZZ/4\ZZ)$ implies that $\rho_{E,4}(\calG_{k(\mu_4)}) = \SL_2(\ZZ/4\ZZ)$.  Thus to prove $\rho_{E,4}(\calG_{k^\cyc}) = \SL_2(\ZZ/4\ZZ)$, it suffices to show that $k(E[4])\cap k^\cyc = k(\mu_4)$.  

In \cite{Lang-Trotter}*{Part III \S11}, it is shown that $\sqrt[4]{\Delta}$ is an element of $k(E[4])$.  Using $\sqrt{\Delta} \not\in k(\mu_4)$, one finds that $k(\mu_4, \sqrt[4]{\Delta})\subseteq k(E[4])$ is an abelian extension of $k(\mu_4)$ of degree $4$.  By Lemma~\ref{L:derived subgroup finite} the group $\SL_2(\ZZ/4\ZZ)^\ab$ is cyclic of order $4$, so $k(E[4]) \cap k(\mu_4)^\ab = k(\mu_4, \sqrt[4]{\Delta})$.  Therefore
\begin{equation} \label{E:4 cyclotomic}
k(E[4]) \cap k^\cyc = (k(E[4])\cap k(\mu_4)^\ab) \cap k^\cyc = k(\mu_4, \sqrt[4]{\Delta}) \cap k^\cyc = k(\mu_4),
\end{equation}
where the last equality uses $\sqrt{\Delta} \not\in k^\cyc$.

\noindent (iii)  The assumption $\rho_{E,9}(\calG_k)\supseteq \SL_2(\ZZ/9\ZZ)$ implies that $\rho_{E,9}(\calG_{k(\mu_9)})= \SL_2(\ZZ/9\ZZ)$.  By Lemma~\ref{L:derived subgroup finite}, the group $\SL_2(\ZZ/9\ZZ)^\ab$ has order $3$.  

Note that $\sqrt[3]{\Delta}$ is an element of $k(E[3])$ (see \cite{Adelmann}*{Proposition 5.4.3} for example).  Arguing as in part (ii), we find that $k(E[9]) \cap k(\mu_9)^\ab = k(\mu_9, \sqrt[3]{\Delta}).$   Since $\sqrt[3]{\Delta} \not\in k^\cyc$, we deduce that $k(E[9]) \cap k^\cyc=k(\mu_9)$, and hence $\rho_{E,9}(\calG_{k^\cyc})=\SL_2(\ZZ/9\ZZ)$.

\noindent (iv) The assumptions imply that $\rho_{E,3}(\calG_k)=\GL_2(\ZZ/3\ZZ)$.  One readily checks that $\GL_2(\ZZ/3\ZZ)'=\SL_2(\ZZ/3\ZZ)$, and thus $\rho_{E,3}(\calG_k)'=\SL_2(\ZZ/3\ZZ)$.   Using (\ref{E:cyclotomic inclusions}), with $m=3$, gives $\rho_{E,3}(\calG_{k^\cyc})=\SL_2(\ZZ/3\ZZ)$.

By Lemma~\ref{L:derived subgroup finite}, the group $\SL_2(\ZZ/9\ZZ)^\ab$ has order $3$.  So from (\ref{E:cyclotomic inclusions}), with $m=9$, we find that $\rho_{E,9}(\calG_{k^\cyc})$ is either $\SL_2(\ZZ/9\ZZ)'$ or $\SL_2(\ZZ/9\ZZ)$.  If $\rho_{E,9}(\calG_{k^\cyc})=\SL_2(\ZZ/9\ZZ)'$, then Lemma~\ref{L:derived subgroup finite} implies that $\rho_{E,3}(\calG_{k^\cyc})\neq \SL_2(\ZZ/3\ZZ)$.  Therefore,  $\rho_{E,9}(\calG_{k^\cyc})=\SL_2(\ZZ/9\ZZ)$.
\end{proof}

We now state a criterion that applies to $k=\QQ$.
\begin{lemma}[Jones]  \label{L:Serre curve criterion}
Let $E$ be an elliptic curve over $\QQ$ which satisfies the following properties:
\begin{alphenum}
\item $\rho_{E,\ell}(\calG_\QQ)\supseteq \SL_2(\ZZ/\ell\ZZ)$ for every prime $\ell\geq 5$,
\item $\rho_{E,72}(\calG_\QQ)\supseteq \SL_2(\ZZ/72\ZZ)$.
\end{alphenum}
Then $E$ is a Serre curve.
\end{lemma}
\begin{proof}
For each $m\geq 1$, we have $\rho_{E,m}(\calG_\QQ)=\GL_2(\ZZ/m\ZZ)$ if and only if $\rho_{E,m}(\calG_\QQ)\supseteq \SL_2(\ZZ/m\ZZ)$.
The lemma is now a special case of  \cite{Jones}*{Lemma~5}. 
\end{proof}

\section{Elliptic curve over finite fields} \label{S:over finite fields}
Fix a positive integer $m$ and a prime $p\nmid m$.  Let $E$ be an elliptic curve over the field $\FF_p$.  As before, one has a Galois representation $\rho_{E,m}\colon \Gal({\FFbar}_p/\FF_p) \to \GL_2(\ZZ/m\ZZ)$, which arises from the Galois action on the $m$-torsion of $E$.  

Let $\Frob_p\in \Gal(\FFbar_p/\FF_p)$ be the $p$-th power Frobenius automorphism.  For a subset $C$ of $\GL_2(\ZZ/m\ZZ)$ stable under conjugation, define the set
\[
\Omega_C(p) := \big\{(r,s) \in \FF_p^2 : \Delta_{r,s} \neq 0, \, \rho_{E(r,s),m}(\Frob_p) \in C \big\}.
\]
The following theorem gives a good estimate on the cardinality of this set.

\begin{thm}[Jones] \label{T:Jones} Fix a positive integer $m$ and a conjugacy class $C$ of $\GL_2(\ZZ/m\ZZ)$.  Let $d$ be the element of $(\ZZ/m\ZZ)^\times$ such that $\det(C)=\{d\}$.  Then for all primes $p$ with $p\equiv d \bmod{m}$,
\[
\frac{|\Omega_C(p)|}{p^2} = \frac{|C| }{|\SL_2(\ZZ/m\ZZ)|}  + O\Bigg(\frac{m^5}{p^{1/2}}\Bigg).
\]
\end{thm}
\begin{proof}
This follows from Theorem~8 (and Theorem~7) of \cite{Jones}.  A key ingredient is a generalization of results of Hurwitz, see \cite{Jones-Trace}.
\end{proof}

\section{The large sieve}
Let $K$ be a number field, $\Lambda$ a free $\OO_K$-module of rank $n$, and $\norm{\cdot}$ a norm on $\Lambda_\RR = \RR\otimes_\ZZ \Lambda$.  Fix a subset $Y$ of $\Lambda$.   Let $x\geq1$ and $Q>0$ be real numbers.  For every prime ideal $\p\in\Sigma_K$, let $\omega_\p$ be a real number in the interval $[0,1)$.  Assume the following conditions hold:
\begin{enumerate}
\item The set $Y$ is contained in a ball of radius $x$; i.e.,~there is an $a_0\in \Lambda_\RR$ such that  $\norm{a-a_0}\leq x$ for all $a\in Y$.
\item  For every $\p\in \Sigma_K(Q)$, the image $Y_\p$ of $Y$ in $\Lambda/\p\Lambda$ by reduction modulo $\p$ satisfies
\[
|Y_\p| \leq (1-\omega_\p)|\Lambda/\p\Lambda|.
\]
\end{enumerate}
\begin{thm}[Large sieve, \cite{SerreMordellWeil}*{\S12.1}] \label{T:large sieve}
With assumptions as above, we have
\[
|Y| \ll_{K,\Lambda, \norm{\cdot}} \frac{ x^{[K:\QQ]n} + Q^{2n}}{L(Q)}
\]
where 
\[
L(Q):=\sum_{ \substack{\aA\subseteq \OO_K \text{squarefree} \\ N(\aA)\leq Q}} \prod_{\p|\aA} \frac{\omega_\p}{1-\omega_\p}.
\]
\end{thm}

\begin{remark}
We will apply the large sieve with $\Lambda=\OO_k^2$ and $\norm{\cdot}$ our fixed norm on $\RR\otimes_\ZZ \OO_k^2$.  In \S\ref{S:exceptional primes} and \S\ref{S:discriminants}, we will take $K$ to be $k$ and $\QQ$, respectively.
\end{remark}

\section{Most elliptic curves have large $\ell$-adic Galois images} \label{S:exceptional primes}
Throughout this section, fix a number field $k$.  

\begin{definition}
For each positive integer $m$, define the set 
\[
B_{k,m}(x):=\{ (a,b) \in B_k(x):  \rho_{E(a,b),m}(\calG_k) \not\supseteq \SL_2(\ZZ/m\ZZ) \}.
\]
\end{definition}
The main goal of this section is to prove the following bound.  
\begin{prop} \label{P:l-adic surjectivity}  There is an absolute constant $\beta\geq 1$ such that
\[
\frac{|B_{k,4}(x) \cup B_{k,9}(x) \cup \bigcup_{\ell \geq 5} B_{k,\ell}(x)|}{|B_k(x)|} \ll_{k,\norm{\cdot}}  \frac{(\log x)^{\beta}}{x^{[k:\QQ]/2}}.
\]  
\end{prop}

\begin{remark}
For an elliptic curve $E$ over $k$, we have Galois representations ${\rho}_{E,\ell^\infty} \colon \calG_k \to \GL_2(\ZZ_\ell)$ coming from the action on the $\ell$-power torsion.
Proposition~\ref{P:l-adic surjectivity} (with Lemma~\ref{L:reduce to finite groups}) shows that for a ``random'' elliptic curve $E$ over $k$, we have ${\rho}_{E,\ell^\infty}(\calG_k)\supseteq \SL_2(\ZZ_\ell)$ for all primes $\ell$.  Since $\det \circ {\rho}_{E,\ell^\infty} \colon \calG_k \to \ZZ_\ell^\times$ is the $\ell$-adic cyclotomic character of $k$, we find that ${\rho}_{E,\ell^\infty}(\calG_k)$ is as ``large as possible'' for all $\ell$. 
\end{remark}

\begin{remark}
Our proof of Proposition~\ref{P:l-adic surjectivity} is clearly based on Duke's paper \cite{Duke}, which proves the $k=\QQ$ case (with Jones \cite{Jones} handling $4$ and $9$).

Unlike Duke's result, the implicit constants in Proposition~\ref{P:l-adic surjectivity} are effective.    The source of non-effective constants in \cite{Duke} is the use of the Siegel-Walfisz theorem.  We avoid this by applying the pigeonhole principle in the proof of Lemma~\ref{L:exceptional prime bound 1} and then sieving only by conjugacy classes with a fixed determinant.
\end{remark}

\subsection{Sieving elliptic curves by Frobenius conjugacy classes}
For a positive integer $m$ and a conjugacy class $C$ of $\GL_2(\ZZ/m\ZZ)$, define the set
\[
Y_C(x) := \{ (a,b) \in B_k(x):  \rho_{E(a,b),m}(\calG_k) \cap C = \emptyset  \big\}.
\]
For $d\in(\ZZ/m\ZZ)^\times$, let $\Sigma_{k}^1(Q;d,m)$ be the set of $\p\in\Sigma_k(Q)$ with degree $1$ (i.e., $N(\p)$ prime) and $N(\p)\equiv d \bmod{m}$. 

\begin{prop}  \label{P:nasty sieving}
Let $m$ be a positive integer and $C$ a conjugacy class of $\GL_2(\ZZ/m\ZZ)$.  Let $d$ be the unique element of $(\ZZ/m\ZZ)^\times$ such that $\det(C)=\{d\}$, and assume that $d\in \chi_k(\calG_k) \bmod{m} \subseteq (\ZZ/m\ZZ)^\times$.  Then
\[
\frac{ |Y_C(x)|}{|B_k(x)|} \ll_{k,\norm{\cdot}}  \Bigg( \frac{|C|}{|\SL_2(\ZZ/m\ZZ)|}  |\Sigma_{k}^1(x^{[k:\QQ]/2};d,m)| + O_k(m^5 x^{[k:\QQ]/4}) \Bigg)^{-1}.
\]
\end{prop}
\begin{proof}  
Let $\Lambda$ be the $\OO_k$-module $\OO_k^2$.  We have already chosen a norm $\norm{\cdot}$ on $\Lambda_\RR:=\RR \otimes_\ZZ \OO_k^2$, and the set $B_k(x)\subseteq \Lambda_\RR$ lies in a ball of radius $x$.  Let $Q:=x^{[k:\QQ]/2}$.

For each $\p\in\Sigma_{k}^1(Q;d,m)$, define
\[
\Omega_\p = \{ (r,s) \in \FF_\p^2 : \Delta_{r,s} \neq 0, \, \rho_{E(r,s),m}(\Frob_{N(\p)})  \in C \}
\]
and $\omega_\p = |\Omega_\p|/N(\p)^2$.  Let $Y_{\p}$ be the image of $Y_C(x)$ in $\FF_\p^2$ via reduction modulo $\p$.   

Suppose that $(a,b)\in B_k(x)$ satisfies $(a,b) \bmod{\p} \in \Omega_\p$; then $E(a,b)$ has good reduction at $\p$ and $\rho_{E(a,b),m}(\Frob_\p) \subseteq C$.  So $\rho_{E(a,b),m}(\calG_k) \cap C \neq \emptyset$, and thus $(a,b) \notin Y_C(x)$.  This shows that
\begin{equation*}
Y_{\p} \subseteq \FF_\p^2 - \Omega_\p,
\end{equation*}
and hence $|Y_{\p}| \leq (1-\omega_\p)|\Lambda/\p\Lambda|.$   

For $\p\notin\Sigma_{k}^1(Q;d,m)$, define $\omega_\p=0$.   By the large sieve (Theorem~\ref{T:large sieve}),
\begin{equation} \label{E: result of large sieve}
|Y_C(x)| \ll_{k,\norm{\cdot}} \frac{x^{2[k:\QQ]} }{L(Q)},
\end{equation}
where 
\[
L(Q):=\sum_{\substack{\aA\subseteq \OO_k \text{ squarefree} \\ N(\aA)\leq Q}} \prod_{\p|\aA}  \frac{\omega_\p}{1-\omega_\p} \geq  \sum_{\p \in \Sigma_{k}^1(Q;d,m)} \omega_\p.
\]
For $\p \in \Sigma_{k}^1(Q;d,m)$, Theorem~\ref{T:Jones} gives
\[
\omega_\p = \frac{|C|}{|\SL_2(\ZZ/m\ZZ)|} + O(m^5/N(\p)^{1/2}).
\]
Therefore
\begin{align*}
L(Q) & \geq  \sum_{\p \in \Sigma_{k}^1(Q;d,m)} \bigg( \frac{|C|}{|\SL_2(\ZZ/m\ZZ)|} + O(m^5/N(\p)^{1/2}) \bigg)\\
&=\frac{|C|}{|\SL_2(\ZZ/m\ZZ)|}  |\Sigma_{k}^1(Q;d,m)| + O_k(m^5 Q^{1/2}).
\end{align*}
The assumption $d\in \chi_k(\calG_k) \bmod{m}$ is needed to guarantees that $L(Q) \gg_{k,m} 1$.  The proposition follows by combining our lower bound of $L(Q)$ and (\ref{E:box asymptotics}) with (\ref{E: result of large sieve}).
\end{proof}

\subsection{Galois image modulo an integer}
The following proposition shows that for a ``random'' elliptic curve $E$ over $k$, $\rho_{E,m}$ has large image.
\begin{prop} \label{P:exceptional prime bound m}
For a positive integer $m$,
\[
\frac{|B_{k,m}(x)|}{|B_{k}(x)|} \ll_{k,\norm{\cdot},m} \frac{\log x}{x^{[k:\QQ]/2}}.
\]
\end{prop}
\begin{proof}
Let $C_1,\dots,C_n$ be the conjugacy classes of $\GL_2(\ZZ/m\ZZ)$ with determinant $1$.
By Lemma~\ref{L:SL condition for m}, we have
$B_{k,m}(x) \subseteq \bigcup_{i=1}^n Y_{C_i}(x)$.  Proposition~\ref{P:nasty sieving} gives 
\[
\frac{|B_{k,m}(x)|}{|B_k(x)|} \leq \sum_{i=1}^n \frac{|Y_{C_i}(x)|}{|B_k(x)|}  \ll_{k,\norm{\cdot}}  \sum_{i=1}^n \Big( \frac{|C_i|}{|\SL_2(\ZZ/m\ZZ)|}  |\Sigma_{k}^1(x^{[k:\QQ]/2};1,m)| + O_k(m^5 x^{[k:\QQ]/4})\Big)^{-1}.
\]
Using $ |\Sigma_{k}^1(x^{[k:\QQ]/2};1,m)|\gg_{k,m} x^{[k:\QQ]/2}/\log x$, we deduce that 
\[
\frac{|B_{k,m}(x)|}{|B_k(x)|} \ll_{k,\norm{\cdot},m} \sum_{i=1}^n  \big( x^{[k:\QQ]/2}/\log x\big)^{-1} \ll_m \frac{\log x}{x^{[k:\QQ]/2}}.
\qedhere
\]
\end{proof}

\subsection{Galois image modulo primes} 
Let $h$ be the absolute logarithmic height on $\PP^1(\Qbar)$.  For an elliptic curve $E$, let $j(E)$ be its $j$-invariant.  The following theorem bounds the number of $\ell$'s that we need to consider (in particular, it gives an effective version of a result of Serre \cite{Serre-Inv72}).
\begin{thm}[Masser-W\"ustholz  \cite{MasserWustholz}] \label{T:MasserWustholz}
Let $E$ be an elliptic curve defined over a number field $k$, and assume that $E$ does not have complex multiplication.  There are positive absolute constants $c$ and $\gamma$ such that if $\ell > c \big( \max\big\{[k:\QQ], h(j(E))\big\} \big)^\gamma$,  then
$\rho_{E,\ell}(\calG_k)  \supseteq \SL_2(\ZZ/\ell\ZZ).$
\end{thm}

\begin{lemma} \label{L:j bound}
If $(a,b) \in B_k(x)$, then $h(j(E(a,b)))\ll_{k,\norm{\cdot}} \log x$.  \end{lemma}
\begin{proof} 
Let $\Sigma_k^\infty$ be the set of archimedean places of $k$.  For each $v\in \Sigma_k^\infty$, let $|\!\cdot\!|_v$ be an absolute value on the completion $k_v$ of $k$ at $v$. 
On $\prod_{v\in\Sigma_k^\infty} k_v^2$, we have a norm
$\norm{ (a_v,b_v)_v }_1 = \sup_{v\in \Sigma_k^\infty} |a_v|_v + \sup_{v\in \Sigma_k^\infty} |b_v|_v.$
Using the natural isomorphism $\RR\otimes_\ZZ \OO_k^2 \cong \prod_{v\in\Sigma_k^\infty} k_v^2$, we may view $\norm{\cdot}_1$ as a norm on $\RR\otimes_\ZZ \OO_k^2$.
Recall that $j(E(a,b)) = - 1728(4a)^3/ \Delta_{a,b}$.  Since $a$ and $b$ are integral, we have
 \begin{align*}
 h(j(E(a,b))) & = h([- 1728(4a)^3 : \Delta_{a,b}] )  \ll_k {\sum}_{v\in \Sigma_k^\infty } \log( \max\{ 1728\cdot 4^3|a|^3_v, |\Delta_{a,b}|_v \} )
 \end{align*}
and thus $h(j(E(a,b)))  \ll_k \log \norm{(a,b)}_1$. The norms $\norm{\cdot}$ and $\norm{\cdot}_1$ are equivalent, so 
\[
h(j(E(a,b))) \ll_k \log \norm{(a,b)}_1 \ll_{k,\norm{\cdot}} \log \norm{(a,b)} \leq \log x.
\qedhere
\]
\end{proof}

\begin{lemma} \label{L:MasserWustholz}
There is a constant $c_k>0$ (depending only on $k$) and an absolute constant $\gamma>0$ such that 
\[
\big\{ (a,b) \in B_k(x):  \rho_{E(a,b),\ell}(\calG_k) \not\supseteq \SL_2(\ZZ/\ell\ZZ)  \text{ for some prime $\ell\geq 5$} \big\} =  \bigcup_{5\leq \ell \leq c_k (\log x)^\gamma} B_{k,\ell}(x).
\]
\end{lemma}
\begin{proof} 
For an elliptic curve $E/k$ with complex multiplication, we have $\rho_{E,\ell}(\calG_k) \not\supseteq \SL_2(\ZZ/\ell\ZZ)$ for every prime $\ell\geq 5$.  The lemma follows by combining Theorem~\ref{T:MasserWustholz} and Lemma~\ref{L:j bound}.  
\end{proof}

\begin{lemma} \label{L:exceptional prime bound 1}
Assume that $5\leq \ell \leq c_k (\log x)^\gamma$, where $c_k$ and $\gamma$ are the constants from Lemma~\ref{L:MasserWustholz}.  Then
\[ 
\frac{|B_{k,\ell}(x)|}{|B_k(x)|} \ll_{k,\norm{\cdot}} \frac{ (\log x)^{7\gamma+1} }{x^{[k:\QQ]/2}}.
\]
\end{lemma}
\begin{proof}
We may assume that $\ell$ satisfies $k\cap\QQ(\mu_\ell)=\QQ$ (this excludes only a finite number of $\ell$, which can be handled with Proposition~\ref{P:exceptional prime bound m}).  
Define the set $\Sigma^1_k(x) := \{ \p \in \Sigma_k(x) : N(\p) \text{ is prime}\}$.
By the pigeonhole principle, there is an element $d\in (\ZZ/\ell\ZZ)^\times = \chi_k(\calG_k) \bmod{\ell}$ such that 
\[
|\Sigma_{k}^1(x^{[k:\QQ]/2};d,\ell)|  \geq \frac{1}{\ell-1}|\Sigma^1_k(x^{[k:\QQ]/2})| +O_k(1) \gg_k \frac{1}{\ell-1} x^{[k:\QQ]/2}/\log x.
\]
Let $C_1,\dots, C_n$ be the conjugacy classes of $\GL_2(\ZZ/\ell\ZZ)$ with $\det(C_i)=d$.  Combining Lemma~\ref{L:SL condition for ell} and Proposition~\ref{P:nasty sieving}, we have
\begin{align*}
\frac{|B_{k,\ell}(x)|}{|B_k(x)|} \leq  \sum_{i=1}^n \frac{|Y_{C_i}(x)|}{|B_k(x)|}
&\ll_{k,\norm{\cdot}} \sum_{i=1}^n \left(\frac{|C_i|}{|\SL_2(\ZZ/\ell\ZZ)|} |\Sigma_{k}^1(x^{[k:\QQ]/2};d,\ell)| + O_k(\ell^5x^{[k:\QQ]/4}) \right)^{-1} \\
&\ll_{k,\norm{\cdot}} \sum_{i=1}^n \left(\frac{|C_i|}{|\GL_2(\ZZ/\ell\ZZ)|} x^{[k:\QQ]/2}/\log x + O_k(\ell^5x^{[k:\QQ]/4})\right)^{-1}.   
\end{align*}
The bounds $n\leq \ell^3$, $1\leq |C_i|$, and $\ell\leq c_k(\log x)^\gamma$ imply
\[
\frac{|B_{k,\ell}(x)|}{|B_k(x)|}  \ll_{k,\norm{\cdot}} n \ell^4 \frac{\log x}{x^{[k:\QQ]/2}} \ll_k \frac{(\log x)^{7\gamma+1}}{x^{[k:\QQ]/2}}. \qedhere
\]
\end{proof}

\subsection{Proof of Proposition~\ref{P:l-adic surjectivity}}
Using Lemmas~\ref{L:MasserWustholz} and \ref{L:exceptional prime bound 1}, we obtain the following bounds: 
\[
\frac{|\bigcup_{\ell \geq 5} B_{k,\ell}(x) |}{|B_k(x)|}  \leq   \sum_{5\leq \ell \leq c_k (\log x)^\gamma} \frac{|B_{k,\ell}(x)|}{|B_k(x)|} 
 \ll _{k,\norm{\cdot}}  \sum_{5\leq \ell \leq c_k (\log x)^\gamma} \frac{(\log x)^{7\gamma+1}}{ x^{[k:\QQ]/2}} 
 \ll_k   \frac{(\log x)^{8\gamma + 1}}{ x^{[k:\QQ]/2}}.
\]
By Proposition~\ref{P:exceptional prime bound m} (with $m=4$ and $9$), we have
\[
\frac{|B_{k,4}(x) \cup B_{k,9}(x)|}{|B_k(x)|}  \ll_{k,\norm{\cdot}} \frac{\log x}{x^{[k:\QQ]/2}}.
\]
The proposition follows immediately with $\beta=8\gamma+1$.

\section{Discriminants} \label{S:discriminants}

\begin{prop} \label{P:Serre obstruction bound}  
Fix a number field $k\neq \QQ$, an integer $r\geq 2$, and assume that $k$ contains $\mu_r$.  Then
\[
\frac{|\{ (a,b) \in B_k(x) :  \sqrt[r]{\Delta_{a,b}} \in k^\cyc  \}|}{|B_k(x)|} \ll_{k,\norm{\cdot},r}  \frac{\log x}{\sqrt{x}}.
\]
\end{prop}
\begin{remark}
From Proposition~\ref{P:Serre obstruction bound}, we find that conditions (\ref{I:c}) and (\ref{I:d}) of Proposition~\ref{P:Criterion} hold for ``most'' elliptic curves over a fixed number field $k\neq \QQ$.
\end{remark}

For the rest of this section, we shall fix $k$ and $r$ as in Proposition~\ref{P:Serre obstruction bound}.  Let $d=[k:\QQ]$.  Let $S$ be the finite set of rational primes which satisfies the following conditions with minimal value $\prod_{p\in S} p$;
\begin{itemize}
\item $S$ contains the primes dividing $6r$, 
\item $S$ contains the primes that are ramified in $k$,
\item $\OO_S$ is a principal ideal domain, where $\OO_S$ is the ring of $S'$-integers of $k$ and $S'=\{\p \in \Sigma_k : \p | p,\, \text{ for some } p\in S\}$.
\end{itemize}
Note that the above choice of $S$ depends only on $k$ and $r$.

\begin{lemma}  \label{L:ramification cyclotomic}
Fix a prime $p \notin S$, an element $\Delta\in k^\times$, and let $\Pp_1\dots,\Pp_n$ be the prime ideals of $\OO_{k(\sqrt[r]{\Delta})}$ lying over $p$.  If $\sqrt[r]{\Delta} \in k^\cyc$, then
\[
e(\Pp_1/p) = \dots = e(\Pp_n/p), 
\]
where $e(\Pp_i/p)$ is the ramification index of $\Pp_i$ over $p$.
\end{lemma}
\begin{proof}
Since $\sqrt[r]{\Delta} \in k^\cyc= \QQ^\cyc\cdot k$,  one can show that there is a field $L \subseteq \QQ^\cyc$ such that $k(\sqrt[r]{\Delta}) = L \cdot k$.   Since $p$ is unramified in $k$, we find that $e(\Pp_i/p) = e(\Pp_i \cap \OO_L/p)$.  The value $e(\Pp_i \cap \OO_L/p)$ is independent of $i$, since $L$ is a Galois extension of $\QQ$.
\end{proof}

\begin{lemma} \label{L:beta description}
Let $\calB \subseteq \OO_S^\times$ be a set of representatives for the cosets of $\OO_S^\times/(\OO_S^\times)^r$.  Then for any $\Delta \in \OO_k$ with $\sqrt[r]{\Delta} \in k^\cyc$, there are $m \in \ZZ$, $\alpha \in \OO_S$, and $\beta \in \calB$ such that $\Delta= m \alpha^r \beta$.
\end{lemma}
\begin{proof}
Fix $\Delta \in \OO_k$ with $\sqrt[r]{\Delta} \in k^\cyc$.  We first show that $\Delta$ can be written in the form $m\alpha^r \beta$, for some $m\in \ZZ$, $\alpha \in \OO_S$, and $\beta \in \OO_S^\times$.  We may assume that $\Delta$ is non-zero.  Since $\OO_S$ is a principal ideal domain, there is an element $\alpha\in \OO_S$ such that $0\leq \ord_\p(\Delta/\alpha^r) < r$ for all $\p \not\in S'$.  

Take any prime $p\notin S$ and let $\p_1,\dots, \p_g$ be the prime ideals of $\OO_S$ lying over $p$.  Suppose that some $\p_i$ divides $\Delta/\alpha^r$ in $\OO_S$.  Since $0< \ord_\p(\Delta/\alpha^r) < r$, we deduce that the extension $k(\sqrt[r]{\Delta})/k$ is ramified at $\p_i$.  By Lemma~\ref{L:ramification cyclotomic}, we find that $k(\sqrt[r]{\Delta})/k$ is ramified at all the primes $\p_1,\dots,\p_g$, and hence $p\OO_S = \p_1\dots \p_g$ divides $\Delta/\alpha^r$ in $\OO_S$.    Dividing by $p$ and repeating the above process, we find that there is an integer $m\geq 1$ such that $\beta:=\Delta/(m\alpha^r)$ is an element of $\OO_S^\times$.  We may assume that $\beta$ is in $\calB$ after multiplying $\alpha$ by an appropriate element of $\OO_S^\times$.
\end{proof}

For each $\beta\in \OO_S^\times$, define the sets
\[
W_\beta := \{ (a,b) \in \OO_k^2:  \Delta_{a,b} = m \alpha^r \beta, \text{ for some } m \in \ZZ,\, \alpha \in \OO_S \}
\] 
and $W_\beta(x) := W_\beta \cap B_k(x)$.  For a set $\calB$ as in Lemma~\ref{L:beta description}, we have
\[
\{ (a,b) \in B_k(x) :  \sqrt[r]{\Delta_{a,b}} \in k^\cyc  \} \subseteq \bigcup_{\beta \in \calB} W_\beta(x).
\]
The set $\calB$ is finite (since the abelian group $\OO_S^\times$ is finitely generated), so
\begin{equation} \label{E: W bound for discriminant}
|\{ (a,b) \in B_k(x) :  \sqrt[r]{\Delta_{a,b}} \in k^\cyc  \}| \ll_{k,r} \max_{\beta \in \OO_S^\times} |W_\beta(x)|.
\end{equation}
Thus to prove Proposition~\ref{P:Serre obstruction bound}, it suffices to find bounds for the functions $|W_\beta(x)|$.

\begin{lemma} \label{L:Weil conjecture bound}
Let $p\nmid 6$ be a prime with $p\equiv 1 \bmod{r}$.  For any $\gamma\in \FF_p^\times$,
\[
|\{(a,b) \in \FF_p^2 :  \Delta_{a,b} = \gamma c^r, \, \text{for some $c\in\FF_p$}\}| =  \frac{1}{r} p^2 + O_{r}(p^{3/2}).
\]
\end{lemma}
\begin{proof}
Fix $\gamma\in \FF_p^\times$.  The equation $\Delta_{a,b} = \gamma c^r$ defines a geometrically irreducible variety $X$ in $\AA^3_{\FF_p}=\Spec(\FF_p[a,b,c])$.  Using the Weil conjectures, we find that
\begin{equation} \label{E:Weil conjecture}
|X(\FF_p)| = p^2 + O_r(p^{3/2})
\end{equation}
(that the implicit constant in (\ref{E:Weil conjecture}) depends only on $r$ can be deduced from \cite{Bombieri}).  

For fixed $(a,b)\in \FF_p^2$, if $ \Delta_{a,b} = \gamma c^r$ has a solution $c\in \FF_p^\times$, then it has exactly $r$ such solutions (this uses the assumption $p\equiv 1\bmod{r}$).  Most solutions have $c\neq 0$, since $|\{(a,b)\in\FF_p^2: \Delta_{a,b}=0\}| \ll p$.  The lemma is now immediate.
\end{proof}

\begin{lemma} \label{L:W mod p bound}
Take any  $\beta \in \OO_S^\times$.  Let $p \not\in S$ be a prime that splits completely in $k$, and let $W_{\beta,p}$ be the image of $W_\beta$ in $\OO_k^2/p\OO_k^2$.  Then
\[
|W_{\beta,p}| \leq \Big( \frac{1}{r^{d-1}} + O_{r,d}(p^{-1/2}) \Big) |\OO_k^2/p\OO_k^2|.
\]
\end{lemma}
\begin{proof}
Let $\p_1,\dots, \p_d \in \Sigma_k$ be the prime ideals lying over $p$.   By the Chinese remainder theorem, we have a natural identification $\OO_k/p\OO_k = \prod_{i=1}^d \FF_{\p_i}$.  Then 
\[
W_{\beta,p} \subseteq \bigcup_{ m \in R } \prod_{i=1}^d \big\{(a,b) \in \FF_{\p_i}^2 : \Delta_{a,b}= m \alpha^r \cdot (\beta \bmod{\p_i}),\, \text{for some } \alpha\in \FF_{\p_i}   \big\},
\]
where the union is over a set of coset representatives $R\subseteq \FF_p^\times$ of $\FF_p^\times/(\FF_p^\times)^r$.   We have $p\equiv 1 \bmod{r}$, since $p$ splits completely in $k$ and by assumption $\mu_r\subseteq k$.  By Lemma~\ref{L:Weil conjecture bound},
\[
|W_{\beta,p}| \leq |R| \big(p^2/r + O_{r}(p^{3/2})\big)^d  = p^{2d}/{r^{d-1}} + O_{r,d}(p^{2d-1/2}).  \qedhere
\]
\end{proof}

\begin{lemma} \label{L:W bound}
For $\beta\in \OO_S^\times$,
${|W_\beta(x)|} \ll_{k,\norm{\cdot},r} {|B_k(x)|} (\log x)/{\sqrt{x}}.$
\end{lemma}
\begin{proof}
Let $I$ be the set of primes $p\not\in S$ that split completely in $k$.  By Lemma~\ref{L:W mod p bound}, for each prime $p \in I$, we have $|W_\beta(x) \bmod{p\OO_k^2}| \leq (1-\omega_p) |\OO_k^2/p\OO_k^2|$, where $\omega_p = 1-1/r^{d-1} + O_{r,d}(p^{-1/2})$.  For $p\notin I$, set $\omega_p=0$.

We may now apply the large sieve.  By Theorem~\ref{T:large sieve} (with $K=\QQ$, $\Lambda=\OO_k^2$, $Q=\sqrt{x}$, and our chosen norm $\norm{\cdot}$ on $\RR\otimes_\ZZ \Lambda$), we have
$|W_\beta(x)| \ll_{k,\norm{\cdot}} { x^{2d}}/{L(\sqrt{x})},$
where 
\[
L(\sqrt{x}) := \sum_{ \substack{J \subseteq I \text{ finite} \\  \prod_{p\in J} p \leq \sqrt{x}} } \prod_{p\in J} \frac{\omega_p}{1-\omega_p}.
\]
Using $r\geq 2$ and $d\geq 2$ (since $k\neq \QQ$),  we have the bound
\[
L(\sqrt{x}) \geq \sum_{ \substack{J \subseteq I \text{ finite} \\  \prod_{p\in J} p \leq \sqrt{x} } } \prod_{p\in J} \Big( 1 + O_{r,d}(p^{-1/2}) \Big) \geq \sum_{ \substack{p\in I \\  p \leq \sqrt{x} } } \Big( 1 + O_{r,d}(p^{-1/2}) \Big).
\]
The set $I$ has positive density in the primes, so $L(\sqrt{x}) \gg_{r,k} \sqrt{x}/\log x$.  The lemma follows by using this bound for $L(\sqrt{x})$ and (\ref{E:box asymptotics}) with our upper bound for $|W_\beta(x)|$.
\end{proof}

\begin{proof}[Proof of Proposition~\ref{P:Serre obstruction bound}]
Apply Lemma~\ref{L:W bound} to the bound (\ref{E: W bound for discriminant}).
\end{proof}

\section{Elliptic curves with maximal Galois action} \label{S:proof}
\subsection{Proof of Theorem~\ref{T:main theorem 2}}
Define the sets
\begin{align*}
Y_1(x ) &= B_{k,4}(x) \cup B_{k,9}(x) \cup {\bigcup}_{\ell\geq 5}B_{k,\ell}(x),\\
Y_2(x) &= \{ (a,b) \in B_k(x) :    \sqrt{\Delta_{a,b}}  \in k^\cyc \}, \\
Y_3(x) &= \{ (a,b) \in B_k(x) :    \mu_3 \subseteq k \text{ and } \sqrt[3]{\Delta_{a,b}}  \in k^\cyc\}.
\end{align*}
By Proposition~\ref{P:Criterion}, we have
$\{(a,b) \in B_k(x): \rho_{E(a,b)}(\calG_k) \neq H_k \}  
\subseteq Y_1(x) \cup Y_2(x)  \cup Y_3(x)$,
and thus 
\[
|\{(a,b) \in B_k(x): \rho_{E(a,b)}(\calG_k) \neq H_k \}| \leq |Y_1(x)| + |Y_2(x)| + |Y_3(x)|.
\]
By Proposition \ref{P:l-adic surjectivity}, we have ${|Y_1(x)|}/{|B_k(x)|} \ll_{k,\norm{\cdot}} {(\log x)^{\beta}}/{x^{[k:\QQ]/2}}$, where $\beta\geq 1$ is an absolute constant.  By Proposition~\ref{P:Serre obstruction bound}, we have 
\[
\frac{|Y_2(x)|}{|B_k(x)|} \ll_{k,\norm{\cdot}}  \frac{\log x}{\sqrt{x}} \text{\quad  and \quad} \frac{|Y_3(x)|}{|B_k(x)|} \ll_{k,\norm{\cdot}}  \frac{\log x}{\sqrt{x}}.
\]
Combining everything together gives:
\[
\frac{|\{(a,b) \in B_k(x): \rho_{E(a,b)}(\calG_k) \neq H_k \}|}{|B_k(x)|} \ll_{k,\norm{\cdot}}  \max\Big\{ \frac{(\log x)^{\beta}}{x^{[k:\QQ]/2}}, \frac{\log x}{\sqrt x} \Big\} \ll \frac{\log x}{\sqrt{x}},
\]
where the last bound uses $k\neq \QQ$.

\subsection{Proof of Theorem~\ref{T:Jones-Serre curves}} \label{SS:Serre curves}
The theorem is easily deduced by combining the criterion of Lemma~\ref{L:Serre curve criterion} with Proposition~\ref{P:l-adic surjectivity} and Proposition~\ref{P:exceptional prime bound m} (with $m=72$).

\appendix
\section{Group theory for $\text{SL}_2$}

In this appendix, we collect several basic facts about the groups $\SL_2(\ZZ/m\ZZ)$.  We will need to pay special attention to the primes $2$ and $3$.
\subsection{Abelianizations}
\begin{lemma} \label{L:derived subgroup finite} 
Let $m$ be a positive integer, and define $b:=\gcd(m,12)$.  Reduction modulo $b$ induces an isomorphism $\SL_2(\ZZ/m\ZZ)^\ab\overset{\sim}{\to} \SL_2(\ZZ/b\ZZ)^\ab$.  The group $\SL_2(\ZZ/m\ZZ)^\ab$ is cyclic of order $b$.
\end{lemma}
\begin{proof}
It is well-known that the group $\PSL_2(\ZZ):=\SL_2(\ZZ)/\{\pm I\}$ has a presentation $\langle A,B : A^2=1, B^3=1\rangle$, thus $\PSL_2(\ZZ)^\ab$ is a cyclic group of order $6$.  Under the quotient map, $\SL_2(\ZZ)'$ surjects on to $\PSL_2(\ZZ)'$, so $\SL_2(\ZZ)^\ab$ has order $6$ or $12$.

For each positive integer $m$, reduction modulo $m$ gives a surjective homomorphism $\SL_2(\ZZ)\twoheadrightarrow \SL_2(\ZZ/m\ZZ)$ (see \cite{Shimura}*{Lemma~1.38}).  
We leave it to the reader to verify that the groups $\SL_2(\ZZ/2\ZZ)^\ab$, $\SL_2(\ZZ/3\ZZ)^\ab$ and $\SL_2(\ZZ/4\ZZ)^\ab$ are cyclic of order $2$, $3$ and $4$ respectively.  We deduce that $\SL_2(\ZZ)^\ab$ is cyclic of order $12$ and that reduction modulo $12$ induces an isomorphism $\SL_2(\ZZ)^\ab \overset{\sim}{\to} \SL_2(\ZZ/12\ZZ)^\ab$.   The lemma is easily deduced from this isomorphism.
\end{proof}

\subsection{Reductions}
\begin{lemma}  \label{L:reduce to finite groups}
Let $\ell$ be a prime, $n\geq 1$ an integer, and $H$ a subgroup of $\SL_2(\ZZ/\ell^n\ZZ)$.
\begin{romanenum}
\item
If $\ell\geq 5$ and the image of $H$ modulo $\ell$ is $\SL_2(\ZZ/\ell\ZZ)$, then $H=\SL_2(\ZZ/\ell^n\ZZ)$.
\item
If $n\geq 2$ and the image of $H$ modulo $\ell^2$ is $\SL_2(\ZZ/\ell^2\ZZ)$, then $H=\SL_2(\ZZ/\ell^n\ZZ)$.
\end{romanenum}
\end{lemma}
\begin{proof}
Part (i) is due to Serre, see \cite{Serre-abelian}*{IV-23 Lemma 3}.  We now prove (ii).  By induction, it suffices to show that for each $r\geq 2$, no \emph{proper} subgroup of $\SL_2(\ZZ/\ell^{r+1}\ZZ)$ reduces modulo $\ell^r$ to the full group $\SL_2(\ZZ/\ell^r\ZZ)$.  Let $G$ be any subgroup of $\SL_2(\ZZ/\ell^{r+1}\ZZ)$ such that $G \bmod{\ell^r} = \SL_2(\ZZ/\ell^r\ZZ)$.   It suffices to show that $G$ contains the abelian group $\mathfrak{s} := \{ A \in \SL_2(\ZZ/\ell^{r+1}\ZZ) :  A \equiv I \bmod{\ell^r} \}$.

The group $\mathfrak{s}$ has a natural structure as a $\SL_2(\ZZ/\ell\ZZ)$-module; i.e., conjugate by any lift to $\SL_2(\ZZ/\ell^{r+1}\ZZ)$.  As an $\SL_2(\ZZ/\ell\ZZ)$-module, $\mathfrak{s}$ is generated by 
\[
I + \ell^r\left(\begin{smallmatrix}0 & 1 \\0 & 0\end{smallmatrix}\right)  \text{ \quad and \quad } I + \ell^r\left(\begin{smallmatrix}0 & 0 \\1 & 0\end{smallmatrix}\right).
\]
Since $G \bmod \ell =\SL_2(\ZZ/\ell\ZZ)$, we find that $G\cap \mathfrak{s}$ is a $\SL_2(\ZZ/\ell\ZZ)$-submodule of $\mathfrak{s}$.  Take any $B \in \{\left(\begin{smallmatrix}0 & 1 \\0 & 0\end{smallmatrix}\right),  \left(\begin{smallmatrix}0 & 0 \\1 & 0\end{smallmatrix}\right)  \}$.  We shall now show that $I+\ell^r B \in G$, which will complete the proof of (ii).
By assumption, there exists a $g\in G$ such that $g \equiv I+\ell^{r-1}B \bmod{\ell^r}$.  Taking $\ell$-th powers, and using $r\geq 2$ and $B^2=0$, we find that $I+\ell^rB = g^\ell \in G$.
\end{proof} 

\begin{remark}
Lemma~\ref{L:reduce to finite groups}(i) is not true for $\ell=2$ and $3$ (see \cite{Serre-abelian}*{IV-28 Exercises 2 and 3}).  
\end{remark}

\subsection{Goursat's lemma}
For a finite group $G$, let $\calJ(G)$ be the set of non-abelian simple groups, up to isomorphism, which occur in some/any composition series of $G$.  

\begin{lemma}[Goursat's lemma] \label{L:Goursat}
Let $G_1,\dots,G_n$ be finite groups, and assume that for each $i\neq j$,  $\calJ(G_i) \cap \calJ(G_j) = \emptyset$ and  $\gcd(|G_i^\ab|,|G_j^\ab|)=1$.  Let $H$ be a subgroup of $G_1\times \cdots \times G_n$ such that $\pr_i(H)=G_i$ for every projection $\pr_i \colon G_1\times \cdots \times G_n  \to G_i$.  Then $H=G_1\times \cdots \times G_n$.
\end{lemma}
\begin{proof}
By induction, we may reduce to the case $n=2$.  Define $N_1 = \pr_1( H \cap (G_1 \times \{1\} ))$ and $N_2 = \pr_2( H \cap ( \{1\}\times G_2 ))$ which are normal subgroups of $G_1$ and $G_2$ respectively.   The image of $H$ in $G_1/N_1 \times G_2/N_2$ is the graph of an isomorphism
\begin{equation} \label{E:Goursat}
G_1/N_1 \cong G_2/N_2;
\end{equation}
this fact is usually called \emph{Goursat's lemma} (see \cite{Ribet-76}*{Lemma~5.2.1}). We deduce that $\calJ(G_1/N_1) \subseteq \calJ(G_1) \cap \calJ(G_2) = \emptyset$,
thus the group $G_1/N_1$ is solvable.  The groups $G_1$ and $G_2$ have no common abelian quotients besides $1$ (this follows from the assumption $\gcd(|G_1^\ab|,|G_2^\ab|)=1$), so from (\ref{E:Goursat}) and the solvability, we deduce that $G_1=N_1$ and $G_2=N_2$.  From the definition of the $N_i$, we find that $H$ contains $\{1\} \times G_2$ and $G_1 \times \{1\}$, hence $H= G_1\times G_2$.
\end{proof}

\begin{lemma}[\cite{LangAlgebra}*{XIII Theorem 8.4}] \label{L:simple groups}
For $\ell\geq 5$, $\PSL_2(\ZZ/\ell\ZZ) := \SL_2(\ZZ/\ell\ZZ)/\{\pm I\}$ is a non-abelian simple group of order $(\ell^3-\ell)/2$.  The groups $\SL_2(\ZZ/2\ZZ)$ and $\SL_2(\ZZ/3\ZZ)$ are solvable.
\end{lemma}

\begin{lemma} \label{L:SL Goursat finite}
Let $m$ and $n$ be relatively prime positive integers and let $H$ be a subgroup of $\SL_2(\ZZ/mn\ZZ)$.  Then $H=\SL_2(\ZZ/mn\ZZ)$ if and only if $H$ surjects onto $\SL_2(\ZZ/m\ZZ)$ and $\SL_2(\ZZ/n\ZZ)$ by reduction modulo $m$ and $n$, respectively.
\end{lemma}
\begin{proof}
Using Lemma~\ref{L:simple groups} and the solvability of $\ell$-groups, we deduce that for any positive integer $d$, $\calJ(\SL_2(\ZZ/d\ZZ)) = \{ \PSL_2(\ZZ/\ell\ZZ) : \ell |d, \, \ell\geq 5\}.$
Since $m$ and $n$ are relatively prime, we have $\calJ(\SL_2(\ZZ/m\ZZ))\cap \calJ(\SL_2(\ZZ/n\ZZ)) = \emptyset$.  By Lemma~\ref{L:derived subgroup finite},
\[
\gcd(|\SL_2(\ZZ/m\ZZ)^\ab|, |\SL_2(\ZZ/n\ZZ)^\ab|) = \gcd(m,n,12)=1.
\]
The lemma is now a direct consequence of Lemma~\ref{L:Goursat}.
\end{proof}

\begin{lemma} \label{L:SL(Zhat) Goursat}
Let $H$ be a closed subgroup of $\SL_2(\Zhat)$.  Then $H=\SL_2(\Zhat)$ if and only if $H \bmod{4} = \SL_2(\ZZ/4\ZZ)$, $H \bmod{9} = \SL_2(\ZZ/9\ZZ)$, and $H \bmod{\ell} = \SL_2(\ZZ/\ell\ZZ)$ for all $\ell\geq 5$.  
\end{lemma}
\begin{proof}
We have $H=\SL_2(\Zhat)$ if and only if $H \bmod{m} = \SL_2(\ZZ/m\ZZ)$ holds for all positive integers $m$.   By Lemmas~\ref{L:reduce to finite groups} and \ref{L:SL Goursat finite}, this equivalent to having $H \bmod{m} = \SL_2(\ZZ/m\ZZ)$ whenever $m$ is $4$, $9$, or a prime $\geq 5$.
 \end{proof}

\subsection{Conjugacy classes with fixed determinant}

\begin{lemma} \label{L:SL condition for ell}
Let $\ell$ be a prime and $H$ a subgroup of $\GL_2(\FF_\ell)$.   Fix an element $d\in \FF_\ell^\times$.  If $H\cap C \neq \emptyset$ for every conjugacy class $C$ of $\GL_2(\FF_\ell)$ with $\det(C)=\{d\}$, then $H\supseteq \SL_2(\FF_\ell)$.
\end{lemma}
\begin{proof}
First suppose that $d=b^2$ for some $b\in \FF_\ell^\times$.  The group $H$ then contains an element conjugate in $\GL_2(\FF_\ell)$ to $b\cdot\left(\begin{smallmatrix} 1 & 1 \\ 0 & 1 \end{smallmatrix}\right)$.  In particular, $|H|\equiv 0\bmod{\ell}$.  By \cite{Serre-Inv72}*{Proposition~15} (which needs the condition $|H|\equiv 0 \bmod{\ell}$), we deduce that either $H$ is contained in a Borel subgroup of $\GL_2(\FF_\ell)$ or $H$ contains $\SL_2(\FF_\ell)$.   If $H$ was contained in a Borel subgroup, then the main hypothesis of the lemma would imply that every semisimple matrix in $\GL_2(\FF_\ell)$ of determinant $d$ is diagonalizable over $\FF_\ell$; which is false.  Therefore, $H\supseteq \SL_2(\FF_\ell)$.

Now suppose that $d$ is not a square in $\FF_\ell^\times$.   Without loss of generality, we may assume that $H$ contains the scalar matrices in $\GL_2(\FF_\ell)$.  Since $d$ is   not a square, we have $\det(H)=\FF_\ell^\times$.  

We shall now show that $H=\GL_2(\FF_\ell)$.  Let $H_d$ be the set of elements of $H$ with determinant $d$.  The main hypothesis of the lemma implies that
\begin{equation} \label{E:Jordan}
\{ A \in \GL_2(\FF_\ell): \det(A)=d\} =  \bigcup_{g\in \GL_2(\FF_\ell)/H} g H_d g^{-1}.
\end{equation}
By counting both sides, we find that the expression (\ref{E:Jordan}) must be a disjoint union.  Therefore
\begin{equation}\label{E:Jordan middle step}
\bigcup_{h \in H_d} \{ g \in \GL_2(\FF_\ell) : g h g^{-1} \in H\} \subseteq H.
\end{equation}
Using (\ref{E:Jordan}) and (\ref{E:Jordan middle step}), we find that $H$ contains both split and non-split Cartan subgroups of $\GL_2(\FF_\ell)$ (see \cite{Serre-Inv72}*{\S2.1} for definitions).  Using Propositions~17 and 14 of \cite{Serre-Inv72}, we deduce that $H=\GL_2(\FF_\ell)$.
\end{proof}

\begin{lemma} \label{L:SL condition for m}
Let $m$ be a positive integer and let $H$ be a subgroup of $\GL_2(\ZZ/m\ZZ)$. 
If $H\cap C \neq \emptyset$ for every conjugacy class $C$ of $\GL_2(\ZZ/m\ZZ)$ with determinant $1$, then $H\supseteq \SL_2(\ZZ/m\ZZ)$.
\end{lemma}
\begin{proof}
By replacing $H$ with $H\cap \SL_2(\ZZ/m\ZZ)$, we may assume that $H$ is a subgroup of $\SL_2(\ZZ/m\ZZ)$.  By Lemma~\ref{L:SL Goursat finite}, it suffices to consider the case where $m$ is a prime power.  By Lemma~\ref{L:reduce to finite groups}, we may further assume that $m$ is $4$, $9$, or a prime.  The case where $m$ is prime, is a consequence of Lemma~\ref{L:SL condition for ell}.  

We may thus assume that $m=\ell^2$, where $\ell=2$ or $3$.  There is an exact sequence
\[
1 \to \mathfrak{s} \to \SL_2(\ZZ/m\ZZ) \overset{\bmod{\ell}}{\to} \SL_2(\ZZ/\ell\ZZ) \to 1.
\]
Since $\mathfrak{s}$ is abelian, it has a natural $\SL_2(\ZZ/\ell\ZZ)$-action; i.e., lift to an element of $\SL_2(\ZZ/m\ZZ)$ and act via conjugate on $\mathfrak{s}$.  By the prime case of the lemma, we find that the image of $H$ modulo $\ell$ is $\SL_2(\ZZ/\ell\ZZ)$.  Therefore $H\cap \mathfrak{s}$ is a normal subgroup of $\SL_2(\ZZ/m\ZZ)$.

If $\ell=3$, then $H \cap \mathfrak{s}$ contains an element conjugate in $\GL_2(\ZZ/9\ZZ)$ to $A:=\left( \begin{smallmatrix} 1 & 3 \\ 3 & 1 \end{smallmatrix}\right)$.  The conjugacy classes of $A$ in $\SL_2(\ZZ/9\ZZ)$ and $\GL_2(\ZZ/9\ZZ)$ are equal and have cardinality 12.  Since $H \cap \mathfrak{s}$ is a normal subgroup of $\SL_2(\ZZ/9\ZZ)$, we deduce that $H\cap\mathfrak{s}$ contains at least $12$ elements.  Since $|\mathfrak{s}|=27$, we conclude that $H\cap\mathfrak{s}=\mathfrak{s}$ and hence $H=\SL_2(\ZZ/9\ZZ)$.

If $\ell=2$, then $H \cap \mathfrak{s}$ contains $B:=\left( \begin{smallmatrix} 3 & 0 \\ 0 & 3 \end{smallmatrix}\right)$ which is in the center of $\GL_2(\ZZ/4\ZZ)$.  Also $H\cap\mathfrak{s}$ must contain at least one element not in $\{I,B\}$.  Since $|\mathfrak{s}|=8$, we have
\[
[\SL_2(\ZZ/4\ZZ): H] = [\mathfrak{s}: H\cap \mathfrak{s}] \in \{1,2\}.
\]
Suppose that $[\SL_2(\ZZ/4\ZZ): H]=2$.  Then $H$ is a normal subgroup of $\SL_2(\ZZ/4\ZZ)$ with quotient cyclic of order $2$.  However, by Lemma~\ref{L:derived subgroup finite} there is only one index $2$ subgroup of $\SL_2(\ZZ/4\ZZ)$, and when reduced modulo $2$, it does not have image $\SL_2(\ZZ/2\ZZ)$.  Therefore $[\SL_2(\ZZ/4\ZZ): H]=1$.
\end{proof}


\end{document}